\documentclass[preprint,11pt]{elsarticle}
\makeatletter
\def\ps@pprintTitle{%
 \let\@oddhead\@empty
 \let\@evenhead\@empty
 \def\@oddfoot{\centerline{\thepage}}%
 \let\@evenfoot\@oddfoot}
\makeatother
\usepackage{amssymb,amsmath,amsthm}
\usepackage{epsfig}
\usepackage{color}
\usepackage{makeidx}
\usepackage{bbm}
\usepackage[top=1in, bottom=1in, inner=1in, outer=1in]{geometry}
\usepackage{setspace}

\newtheorem{theorem}{Theorem}
\newtheorem*{thma}{Theorem A}
\newtheorem*{thmb}{Theorem B}

\newtheorem*{propa}{Proposition A}

\newtheorem{definition}{Definition}
\newtheorem{lemma}{Lemma}

\newtheorem*{remark}{Remark}
\newcommand*\xbar[1]{%
  \hbox{%
    \vbox{%
      \hrule height 0.5pt 
      \kern0.4ex
      \hbox{%
        \kern-0.15em
        \ensuremath{#1}%
        \kern-0.15em
      }%
    }%
  }%
}

\begin{document}
\begin{frontmatter}

\title{On the volume of the shrinking branching Brownian sausage}

\author{Mehmet \"{O}z}
\ead{mehmet.oz@ozyegin.edu.tr}
\ead[url]{https://faculty.ozyegin.edu.tr/mehmetoz/}

\address{Department of Natural and Mathematical Sciences, Faculty of Engineering, \"{O}zye\u{g}in University, Istanbul, Turkey}

\begin{abstract}
The branching Brownian sausage in $\mathbb{R}^d$ was defined by Engl\"ander in [Stoch.\ Proc.\ Appl.\ 88 (2000)] similarly to the classical Wiener sausage, as the random subset of $\mathbb{R}^d$ scooped out by moving balls of fixed radius with centers following the trajectories of the particles of a branching Brownian motion (BBM). We consider a $d$-dimensional dyadic BBM, and study the large-time asymptotic behavior of the volume of the associated exponentially shrinking branching Brownian sausage (BBM-sausage). Using a previous result on the density of the support of BBM, and some well-known results on the classical Wiener sausage and Brownian hitting probabilities, we obtain almost sure limit theorems as time tends to infinity on the volume of the shrinking BBM-sausage in all dimensions. 
\end{abstract}

\vspace{3mm}

\begin{keyword}
Branching Brownian motion \sep density \sep sausage \sep strong law of large numbers
\vspace{3mm}
\MSC[2010] 60J80 \sep 60F15 \sep 60D05 \sep 92D25
\end{keyword}

\end{frontmatter}

\pagestyle{myheadings}
\markright{Volume of BBM sausage\hfill}

\section{Introduction}\label{intro}

\subsection {Formulation of the problem and background}

Let $X=(X(t))_{t\geq 0}$ be a standard $d$-dimensional Brownian motion starting at the origin. The Wiener sausage of radius $r$ associated to $X$ is the set-valued process defined by 
$$X^r(t)=\bigcup_{0\leq s\leq t}B(X(s),r),$$
where $B(x,r)$ is the closed ball of radius $r>0$ centered at $x\in\mathbb{R}^d$. For each $t\geq 0$, $X^r(t)$ is then a random subset of $\mathbb{R}^d$, which looks like a `sausage' scooped out over the period $[0,t]$ by a moving ball of fixed radius centered at the Brownian trajectory. Note that the Wiener sausage is a non-Markovian functional of $X$.  

Now let $Z=(Z(t))_{t\geq 0}$ be a $d$-dimensional strictly dyadic branching Brownian motion (BBM). The process can be described as follows. It starts with a single particle at the origin, which performs a Brownian motion in $\mathbb{R}^d$ for a random lifetime, at the end of which it dies and simultaneously gives birth to two offspring. Similarly, starting from the position where their parent dies, each offspring particle repeats the same procedure as their parent independently of each other and of the parent, and the process evolves through time in this way. All particle lifetimes are exponentially distributed with constant parameter $\beta>0$, which is called the branching rate. For each $t\geq 0$, $Z(t)$ can be viewed as a finite discrete measure on $\mathbb{R}^d$, which is supported at the particle positions at that time. We use $P$ and $E$, respectively, as the probability and corresponding expectation for $Z$. The \emph{range} (accumulated support) of $Z$ is the process defined by
$$R(t)=\bigcup_{0\leq s\leq t} \text{supp}(Z(s)),$$
and the \emph{branching Brownian sausage (BBM-sausage) with radius $r$} associated to $Z$ is the process defined by
\begin{equation}
\widehat{Z}_t^r:=\bigcup_{x\in R(t)}B(x,r). \nonumber
\end{equation}  

The Wiener sausage and various set functions of it, especially its volume, have been frequently studied going back to \cite{S1964}. In \cite{DV1975}, Donsker and Varadhan obtained an asymptotic result on the Laplace transform of the volume of the Wiener sausage, which is a large-deviation (LD) result giving information on the probability that the volume is aytpically small. In \cite{B1990} and \cite{S1990}, the work in \cite{DV1975} was extended to the case of the so-called shrinking Wiener sausage. We refer the reader to \cite[Sect.\ 1]{H2010} and the references therein for a brief survey of limit theorems on the volume of the Wiener sausage. 

The branching Brownian sausage was introduced by Engl\"ander in \cite{E2000} in analogy with the classical Wiener sausage, and an asymptotic result on the Laplace transform of its volume was obtained similar to the one in \cite{DV1975}, by using an equivalence to a trapping problem of BBM among Poissonian traps. In more detail, consider a Poisson point process $\Pi$ on $\mathbb{R}^d$ with intensity measure $\nu$, and for $r>0$ define the random trap field as
$$K:=\bigcup_{x_i\,\in\,\text{supp}\Pi}B(x_i,r).$$   
Define the first trapping time of the BBM as $T:=\inf\left\{t\geq 0:R(t)\cap K\neq\emptyset\right\}$, and the event of survival of BBM from traps up to time $t$ as $\mathsf{S_t}:=\left\{T>t\right\}$. Then, denoting the annealed law of the traps and the BBM as $\mathbb{P}$, the first trapping problem of BBM among a Poissonian field of traps in $\mathbb{R}^d$ is related to the BBM-sausage by Fubini's theorem:
\begin{equation}
\mathbb{P}(\mathsf{S_t})=E\left[e^{-\nu\left(\widehat{Z}^r_t\right)}\right]. \nonumber
\end{equation}
Note that $\nu(\widehat{Z}^r_t)$ is proportional to the volume of $\widehat{Z}^r_t$ if the trap intensity is uniform. To the best of our knowledge, apart from \cite{E2000}, no further work was done on the volume of the BBM-sausage.

\subsection{Motivation}

The current work can be regarded as a sequel to the recent works \cite{M2018} and \cite{M2019} under the common theme of \emph{spatial distribution of mass in BBM.} In \cite{M2018}, the mass of BBM falling in linearly moving balls of fixed radius was studied, and an LD result on the large-time probability that this mass is atypically small on an exponential scale was obtained. The asymptotics of the probability of absence of BBM in linearly moving balls of fixed radius, emerged as a special case \cite[Corollary 2]{M2018}. It is well-known that the total mass of BBM typically grows exponentially in time. In \cite{M2019}, the following was asked: how homogeneously are the exponentially many particles at time $t$ spread out over a \emph{subcritical ball}? (Please see \cite[Def.\ 1]{M2019}.) This homogeneity question was formulated in terms of the degree of density of support of BBM at time $t$. Firstly, \cite[Corollary 2]{M2018} was extended to the case of the mass falling in linearly moving balls of exponentially shrinking radius $r(t)=r_0e^{-kt}$, and then via a covering by sufficiently many of such balls, an LD result concerning the $r(t)$-density of the support of BBM in subcritical balls was obtained. The concept of $r(t)$-density of the support of BBM naturally led to the following definition.
\begin{definition}[Enlargement of BBM]
Let $Z=(Z(t))_{t\geq 0}$ be a BBM. For $t\geq 0$, we define the \emph{$r$-enlargement} of $Z(t)$ as
$$Z_t^r:=\bigcup_{x\,\in\,\emph{supp}(Z(t))}B(x,r).$$
\end{definition}
In \cite{M2019}, furthermore, the following results were obtained on the large-time behavior of $r(t)$-enlargement of BBM in $\mathbb{R}^d$. Theorem A below says that, with probability one, an $r(t)$-enlargement of BBM with an exponentially decaying $r$ covers corresponding subcritical balls eventually provided that the exponential decay rate is not too big. Theorem B is on the large-time behavior of the volume of $r(t)$-enlargement of BBM. 
\begin{thma}[Almost sure density of BBM; \cite{M2019}] 
Let $0<\theta<1$, $0\leq k<(1-\theta^2)/d$, $r_0>0$ and $r:\mathbb{R}_+\rightarrow\mathbb{R}_+$ be defined by $r(t)=r_0 e^{-\beta kt}$. For $t>0$ define $\rho_t:=\theta\sqrt{2\beta}t$. Then, in any $d\geq 1$,
$$P(\Omega_0)=1,\:\:\text{where}\quad\Omega_0:=\left\{\omega: \exists \: t_0=t_0(\omega) \: \text{such that} \:\: \forall \: t\geq t_0, \: B(0,\rho_t)\subseteq Z^{r_t}_t(\omega)\right\}.$$
\end{thma}
Throughout the manuscript, for a Borel set $A\subseteq \mathbb{R}^d$, we say \emph{volume} of $A$ to refer to its Lebesgue measure, which we denote by $\textsf{vol}(A)$, and we use $\omega_d$ to denote the volume of the $d$-dimensional unit ball. 
\begin{thmb}[Almost sure volume of enlargement of BBM; \cite{M2019}]
Let $0\leq k\leq 1/d$, $r_0>0$ and $r:\mathbb{R}_+\rightarrow\mathbb{R}_+$ be defined by $r(t)=r_0 e^{-\beta kt}$. Then, with probability one,
\begin{equation}\underset{t\rightarrow\infty}{\lim}\frac{\textsf{vol}\left(Z^{r_t}_t\right)}{t^d}=[2\beta(1-kd)]^{d/2}\omega_d. \nonumber
\end{equation}
\end{thmb}
Motivated by the results above, we ask the following question in the present work. For large $t$, by how much on the scale of $t^d$, if at all, is the volume of the BBM-sausage with radius $r(t)$ (i.e., the $r(t)$-shrinking BBM-sausage) larger than that of the $r(t)$-enlargement of BBM? The aim here is to answer this question in a precise way as $t\rightarrow\infty$.  
 
\smallskip

{ \bf Notation:} We introduce further notation for the rest of the manuscript. For $x\in\mathbb{R}^d$, we use $|x|$ to denote its Euclidean norm. We use $c, c_0, c_1,\ldots$ as generic positive constants, whose values may change from line to line. If we wish to emphasize the dependence of $c$ on a parameter $p$, then we write $c_p$ or $c(p)$. We use $\mathbb{R}_+$ to denote the set of nonnegative real numbers, and write $o(f(t))$ to refer to $g(t)$, where $g:\mathbb{R}_+\to\mathbb{R}_+$ is a generic function satisfying $g(t)/f(t)\rightarrow 0$ as $t\rightarrow\infty$, unless otherwise stated. Also, for a function $g:\mathbb{R}_+\to\mathbb{R}_+$, we use $g_t=g(t)$ for notational convenience. We denote by $X=(X(t))_{t\geq 0}$ a generic standard Brownian motion in $d$-dimensions, and use $\mathbf{P}_x$ and $\mathbf{E}_x$, respectively, as the law of $X$ started at position $x\in\mathbb{R}^d$, and the corresponding expectation. 

\smallskip

{ \bf Outline:} The rest of the paper is organized as follows. In Section 2, we present our main results. In Section 3, we develop the preparation needed for the proofs of our main results, and then give the heuristic argument behind them. The proofs of the main results are given in Section 4.

\section{Results}

Theorem~\ref{theorem1} and Theorem~\ref{theorem2} are on the almost sure growth of exponentially shrinking BBM-sausages in $d=2$ and $d\geq 3$, respectively. 

Let $k\geq 0$, $r_0>0$, and $r:\mathbb{R}_+\rightarrow\mathbb{R}_+$ be defined by $r(t)=r_0 e^{-\beta kt}$.
\begin{theorem}\label{theorem1} 
In $d=2$, with probability one,
\begin{equation} 
\underset{t\rightarrow\infty}{\lim}\frac{\textsf{vol}\left(\widehat{Z}^{r_t}_t\right)}{t^2}= 2\pi \beta. \nonumber
\end{equation}
\end{theorem}

\begin{theorem}\label{theorem2} 
In $d\geq 3$, with probability one,
\begin{equation} 
\underset{t\rightarrow\infty}{\lim}\frac{\textsf{vol}\left(\widehat{Z}^{r_t}_t\right)}{t^d}
= \begin{cases}   
[2\beta(1-k(d-2))]^{d/2}\omega_d & \text{if} \:\: k<1/(d-2), \\
0 & \text{if} \:\: k\geq 1/(d-2).
\end{cases} \label{thm2}
\end{equation}
\end{theorem}

\begin{remark}
Theorem~\ref{theorem1} says that in $d=2$, the large-time behavior of the volume of $\widehat{Z}^{r_t}_t$ is as different as it can be from that of $Z^{r_t}_t$, which is given by Theorem B as: with probability one,  
\begin{align}
\underset{t\rightarrow\infty}{\lim}\frac{\textsf{vol}\left(Z^{r_t}_t\right)}{t^2}&= \begin{cases} 2\pi\beta(1-2k) & \text{if} \:\: k<1/2, \\
0 & \text{if} \:\: k\geq 1/2.
\end{cases} \nonumber
\end{align}
This can be explained as follows. In $d=2$, the motion component of BBM plays a dominating role in the large-time behavior of the shrinking BBM sausage due to the almost sure neighborhood recurrence of Brownian motion. Note that the result does not depend on $k$. For large $t$, a BBM-sausage with any exponentially shrinking radius (independent of how large the exponential rate of decay is for the radius) covers the entire \emph{subcritical zone}, that is, for any $0<\varepsilon<1$, the sausage $\widehat{Z}^{r_t}_t$ eventually covers $B(0,\sqrt{2\beta}(1-\varepsilon)t)$ almost surely.

In $d\geq 3$, from Theorem B, we have  
\begin{equation}\underset{t\rightarrow\infty}{\lim}\frac{\textsf{vol}\left(Z^{r_t}_t\right)}{t^d}= \begin{cases}   
[2\beta(1-kd)]^{d/2}\omega_d & \text{if} \:\: k\leq 1/d, \\
0 & \text{if} \:\: k>1/d.
\end{cases} \nonumber
\end{equation}
Hence, Theorem~\ref{theorem2} says that in $d\geq 3$ for large $t$, provided that the decay of the sausage radius is slow enough, the accumulated support of BBM over $[0,t)$ has a non-trivial contribution to the volume of the $r_t$-shrinking sausage over $[0,t]$ although the contribution is not significant enough to cover the entire subcritical zone; whereas, if the decay of $r_t$ is sharper than a certain threshold (i.e., if $k>1/(d-2)$), the accumulated support over $[0,t)$ and the support at time $t$ both have negligible contribution on the scale of $t^d$. 
\end{remark}

\begin{remark}
In $d=1$, it is clear that with probability one,
\begin{equation} \underset{t\rightarrow\infty}{\lim}\frac{\textsf{vol}\left(\widehat{Z}^{r_t}_t\right)}{t}=\underset{t\rightarrow\infty}{\lim}\frac{\left[2\,r_t+\underset{0\leq s\leq t}{\sup}R(t)-\underset{0\leq s\leq t}{\inf}R(t)\right]}{t}=2\sqrt{2\beta}, \nonumber
\end{equation}
which follows from the well-known result of Bramson \cite{B1978} that the speed of strictly dyadic BBM converges to $\sqrt{2\beta}$ as $t\rightarrow\infty$ with probability one. On the other hand, Theorem B says that with probability one,  
\begin{equation}
\underset{t\rightarrow\infty}{\lim}\frac{\textsf{vol}\left(Z^{r_t}_t\right)}{t}=2\sqrt{2\beta(1-k)}. \nonumber
\end{equation}
Therefore, the large-time behavior of the volume of $\widehat{Z}^{r_t}_t$ is as different as it can be from that of $Z^{r_t}_t$.
\end{remark}

\section{Preparations}

\subsection{Preliminary results}

In this section, we develop preparatory results for the proofs of Theorem~\ref{theorem1} and Theorem~\ref{theorem2}. The first result is about the large-time asymptotic probability of atypically large Brownian displacements. For a proof, see for example \cite[Lemma 5]{OCE2017}. As before, let $X=(X(t))_{t\geq 0}$ be a generic standard Brownian motion in $d$-dimensions, and $\mathbf{P}_x$ the law of $X$ started at position $x\in\mathbb{R}^d$, with corresponding expectation $\mathbf{E}_x$. 

\begin{propa}[Linear Brownian displacements]\label{proposition3}
For $\gamma>0$,
\begin{equation}  \mathbf{P}_0\left(\underset{0\leq s\leq t}{\sup}|X(s)|>\gamma t\right)=\exp[-\gamma^2t/2+o(t)].  \nonumber 
\end{equation}
\end{propa}

\smallskip

The behavior as $t\rightarrow\infty$ of the expected volume of the Wiener sausage is well-known (\cite{S1964}, \cite{L1988}): 
\begin{equation} \mathbf{E}_0[\textsf{vol}(X^r_t)]
= \begin{cases}
\sqrt{\frac{8t}{\pi}}(1+o(1)), & d=1, \\
\frac{2\pi t}{\log t}(1+o(1)), & d=2, \\
\kappa_r t(1+o(1)), & d\geq 3,
\end{cases}  \label{eq201}
\end{equation}
where $r>0$ is constant, and $\kappa_r=r^{d-2}2\pi^{d/2}/\Gamma(d/2-1)$ is the Newtonian capacity of $B(0,r)$ with $\Gamma$ denoting the gamma function. The expected volume of the shrinking Wiener sausage can be obtained from \eqref{eq201} and the scaling invariance of Brownian motion.

\begin{lemma}[Expected volume of exponentially shrinking Wiener sausage] \label{lemma1}
Let $k>0$, $r_0>0$, and $r:\mathbb{R}_+\rightarrow\mathbb{R}_+$ be defined by $r(t)=r_0 e^{-\beta kt}$. Then,
\begin{equation} \mathbf{E}_0[\emph{\textsf{vol}}(X^{r_t}_t)]
= \begin{cases}
\sqrt{\frac{8t}{\pi}}(1+o(1)), & d=1, \\
\frac{\pi}{\beta k}(1+o(1)), & d=2, \\
\kappa_{r_0} t e^{-(d-2)\beta kt}(1+o(1)), & d\geq 3.
\end{cases}  \nonumber
\end{equation}
\end{lemma}

\begin{proof} Write
\begin{equation} \bigcup_{0\leq s\leq t}B(X(s),r_0 e^{-\beta kt})=e^{-\beta kt}\bigcup_{0\leq s\leq t}B(X(s)e^{\beta kt},r_0). \label{eq203}
\end{equation}
By scaling invariance, we have
\begin{equation} \mathbf{E}_0\left[\textsf{vol}\left(\bigcup_{0\leq s\leq t}B(X(s)e^{\beta kt},r_0)\right)\right]=
\mathbf{E}_0\left[\textsf{vol}\left(\bigcup_{0\leq s\leq te^{2\beta kt}}B(X(s),r_0)\right)\right]. \label{eq204}
\end{equation}
Then, it follows from \eqref{eq203} and \eqref{eq204} that
\begin{equation} \mathbf{E}_0\left[\textsf{vol}\left(\bigcup_{0\leq s\leq t}B(X(s),r_0 e^{-\beta kt})\right)\right]=e^{-\beta (kd)t} \mathbf{E}_0\left[\textsf{vol}\left(\bigcup_{0\leq s\leq te^{2\beta kt}}B(X(s),r_0)\right)\right]. \label{eq205}
\end{equation}
Since $te^{2\beta kt}\rightarrow \infty$ as $t\rightarrow\infty$, it follows from \eqref{eq201} that
\begin{equation} \mathbf{E}_0\left[\textsf{vol}\left(\bigcup_{0\leq s\leq te^{2\beta kt}}B(X(s),r_0)\right)\right]
= \begin{cases}
\sqrt{\frac{8t\,e^{2\beta kt}}{\pi}}(1+o(1)), & d=1, \\
\frac{2\pi t\,e^{2\beta kt}}{\log(t\,e^{2\beta kt})}(1+o(1)), & d=2, \\
\kappa_{r_0} t e^{2\beta kt}(1+o(1)), & d\geq 3.
\end{cases}  \label{eq206}
\end{equation}
Use \eqref{eq205} and \eqref{eq206} to complete the proof.
\end{proof}

\begin{lemma}[Hitting probability of exponentially shrinking ball from outside] \label{lemma2}
Let $k>0$, $r_0>0$, and $r:\mathbb{R}_+\rightarrow\mathbb{R}_+$ be defined by $r(t)=r_0 e^{-\beta kt}$. Fix $\rho\in\mathbb{R}^d$ such that $|\rho|=:R>r_0$. Then,
\begin{equation} \mathbf{P}_{\rho}\left(\underset{0\leq s\leq 1}{\min}|X(s)|<r_t\right)=\frac{c_{R}}{\beta kt}(1+o(1))\quad \text{in}\:\:d=2, \label{hitting2}
\end{equation} 
where $c_R=\int_{R^2/2}^\infty \frac{e^{-x}}{2x}\text{d}x$, and
\begin{equation} \mathbf{P}_{\rho}\left(\underset{0\leq s\leq 1}{\min}|X(s)|<r_t\right)=ce^{-\beta k(d-2)t}(1+o(1))\quad \text{in}\:\:d\geq 3, \label{hitting3}
\end{equation}
where $c=c(d,r_0,R)>0$.
\end{lemma}

\begin{proof} 
For $d=2$, from \cite[Eq.\ (1.6)]{S1958} we have
\begin{equation}
\underset{r\rightarrow 0}{\lim}\log\left(\frac{1}{r}\right)\mathbf{P}_{\rho}\left(\underset{0\leq s\leq t}{\min}|X(s)|<r\right)=\int_{R^2/(2t)}^\infty \frac{e^{-x}}{2x}\,\text{d}x. \nonumber
\end{equation} 
Set $t=1$, and then replace $r$ by $r_t$ in the formula above. Then, \eqref{hitting2} follows since $r_t\rightarrow 0$ as $t\rightarrow\infty$. 

For $d\geq 3$, we follow an argument similar to the one in \cite{S1958}, which was developed for the case $d=2$. Recall that $R=|\rho|$ is fixed, and define $G$ and its Laplace transform $\widehat{G}$, respectively, as
$$ G(r,t;R)=\mathbf{P}_{\rho}\left(\underset{0\leq s\leq t}{\min}|X(s)|<r\right),\quad \widehat{G}(r,\lambda;R)=\int_0^\infty e^{-\lambda t}G(r,t;R)dt.$$
It is known that 
\begin{equation}
\widehat{G}(r,\lambda;R)=\frac{1}{\lambda}\left(\frac{r}{R}\right)^v\frac{K_v(\sqrt{2\lambda}R)}{K_v(\sqrt{2\lambda}r)}, \label{lemma2eq1}
\end{equation}
where $v=(d-2)/2$, and $K_v$ is the modified Bessel function of the second kind of order $v$, and that for $d\geq 3$ (i.e., $v>0$), 
\begin{equation}
K_v(z)=\frac{\Gamma(v)}{2}\left(\frac{2}{z}\right)^v (1+o(1)),\quad z\rightarrow 0. \label{lemma2eq2}
\end{equation} 
Then, it follows from \eqref{lemma2eq1} and \eqref{lemma2eq2} that
\begin{equation}
\underset{r\rightarrow 0}{\lim}\frac{\widehat{G}(r,\lambda;R)}{r^{2v}}=\frac{2^{1-v/2}}{\Gamma(v)\,R^v}\,\lambda^{v/2-1}K_v(\sqrt{2\lambda}R). \label{lemma2eq3}
\end{equation}  
Inverting the Laplace transforms in \eqref{lemma2eq3} yields
\begin{equation}
\underset{r\rightarrow 0}{\lim}\frac{G(r,t;R)}{r^{2v}}=\frac{1}{2^v\Gamma(v)}\int_0^t\frac{e^{-\frac{R^2}{2x}}}{x^{1+v}}\text{d}x. \label{lemma2eq4}
\end{equation}
Setting $t=1$, and then replacing $r$ by $r_t$ in \eqref{lemma2eq4} completes the proof of \eqref{hitting3} since $r_t\rightarrow 0$ as $t\rightarrow\infty$.
\end{proof}
 




\subsection{Heuristics}

It is clear that since the \emph{speed} of BBM is $\sqrt{2\beta}$, typically the volume of the BBM-sausage even with constant radius at time $t$ is not larger than $(\sqrt{2\beta})^d\,\omega_d$ on the scale of $t^d$. On the other hand, we know from Lemma~\ref{lemma1} that in $d=2$, the expected volume of the $r_t$-shrinking Wiener sausage is asymptotically constant. Therefore, since there are typically $e^{\beta t+o(t)}$ particles at time $t$, of which at least $e^{\varepsilon t}$ for some $\varepsilon>0$ can be treated as independent particles over the second half of the interval $[0,t]$, and since the volume of the subcritical zone grows only polynomially in $t$, provided that the particles of BBM spread out sufficiently homogeneously over the subcritical zone (see Theorem A), we expect the $r_t$-shrinking BBM-sausage to cover the entire subcritical zone in $d=2$.

The situation is different in $d\geq 3$. By Lemma~\ref{lemma1}, the expected volume of the $r_t$-shrinking Wiener sausage decays like $\exp[-\beta k(d-2)t+o(t)]$. On the other hand, the mass outside $B(0,\theta\sqrt{2\beta}t)$ typically grows like $\exp[\beta t(1-\theta^2)+o(t)]$. Therefore, even if we suppose, for an upper bound on the volume of the BBM-sausage, that all of the particles present outside $B(0,\theta\sqrt{2\beta}t)$ at time $t$ has followed independent Brownian paths over $[0,t]$, and that their respective $r_t$-shrinking Wiener sausages (going back to the initial ancestor) are all disjoint from one another, typically their union could only have a non-trivial volume (on the scale of $t^d$) outside $B(0,\theta\sqrt{2\beta}t)$ provided that $1-\theta^2\geq k(d-2)$, which is equivalent to $\theta\leq \sqrt{1-k(d-2)}$. This explains why the right-hand side of \eqref{thm2} is an upper bound. A heuristic argument as to why it is also a lower bound, can be given similar to the argument above for $d=2$.

\section{Proof of main results}

We first give elementary bounds on $\textsf{vol}(\widehat{Z}_t^{r_t})$ that are valid in any dimension. 

Since $Z_t^{r_t}\subseteq \widehat{Z}_t^{r_t}$ by definition, it is clear that 
$\textsf{vol}\left(Z_t^{r_t}\right)\leq \textsf{vol}(\widehat{Z}_t^{r_t}),$
and therefore, Theorem B implies that with probability one,
\begin{equation}
\underset{t\rightarrow\infty}{\liminf}\:\frac{\textsf{vol}(\widehat{Z}^{r_t}_t)}{t^d}\geq [2\beta(1-kd)]^{d/2}\omega_d. \label{eq51}
\end{equation}
In the rest of the manuscript, let $\mathcal{N}_t$ denote the set of particles of $Z$ that are alive at time $t$, and set $N_t=|\mathcal{N}_t|$. For $u\in\mathcal{N}_t$, let $(Y_u(s))_{0\leq s\leq t}$ denote the ancestral line up to $t$ of particle $u$. By the \emph{ancestral line up to $t$} of a particle present at time $t$, we mean the continuous trajectory traversed up to $t$ by the particle, concatenated with the trajectories of all its ancestors including the one traversed by the initial particle. Note that $(Y_u(s))_{0\leq s\leq t}$ is identically distributed as a Brownian trajectory $(X(s))_{0\leq s\leq t}$ for each $u\in\mathcal{N}_t$. For $t>0$, let $M_t:=\inf \{r\geq 0:R(t)\subseteq B(0,r)\}$. Then, using the union bound, for $\gamma>0$,
\begin{equation} P\left(M(t)>\gamma t\right)= P\left(\exists u\:\in \mathcal{N}_t:\sup_{0\le s\le t}|Y_u(s)|>\gamma t\right) \le E[N(t)]\:\mathbf{P}_0\left(\sup_{0\le s\le t}|X(s)|>\gamma t\right). \label{eq52}
\end{equation}
It is a standard result that $E[N(t)]=\exp(\beta t)$ (see for example \cite[Sect.\ 8.11]{KT1975}). Moreover, we know from Proposition A that $\mathbf{P}_0\left(\sup_{0\le s\le t}|X(s)|>\gamma t\right)=\exp[-\gamma^2 t/2+o(t)]$. Then, for fixed $\varepsilon>0$, defining the events  
$$A_k:=\left\{\textsf{vol}(\widehat{Z}^{r_k}_k)/k^d>[2\beta(1+\varepsilon)]^{d/2}\omega_d\right\},$$
it follows from setting $\gamma=\sqrt{2\beta(1+\varepsilon)}$ in \eqref{eq52} that there exists a positive constant $c(\varepsilon)$ such that $P(A_k)\leq e^{-\beta c(\varepsilon)k}$ for all large $k$. Applying Borel-Cantelli lemma on the events $(A_k:k\geq 1)$, and then choosing $\varepsilon=1/n$, and finally letting $n$ vary over $\mathbb{N}$ yields: with probability one,   
\begin{equation}
\underset{t\rightarrow\infty}{\limsup}\:\frac{\textsf{vol}(\widehat{Z}^{r_t}_t)}{t^d}\leq (2\beta)^{d/2}\omega_d. \label{eq40}
\end{equation}

The proofs below will `close the gap' between $[2\beta(1-kd)]^{d/2}\omega_d$ in \eqref{eq51} and $(2\beta)^{d/2}\omega_d$ in \eqref{eq40} for $d=2$ and $d\geq 3$, separately. When $k=0$, observe that the lower bound in \eqref{eq51} coincides with the upper bound in \eqref{eq40}, so there is nothing more to prove. Therefore, in what follows, we suppose that $k>0$.

\subsection{Proof of Theorem~\ref{theorem1}}

Note that $\omega_2=\pi$. The upper bound comes from \eqref{eq40}. We will show that for every $\varepsilon>0$ there exists a positive constant $c_1$ such that for all large $t$, 
\begin{equation} 
P\left(\frac{\textsf{vol}\left(\widehat{Z}_t^{r_t}\right)}{t^2}\leq 2\pi\beta(1-\varepsilon)\right)\leq e^{-c_1 t}. \label{eq301}
\end{equation}
Then, the lower bound for Theorem~\ref{theorem1} will follow from \eqref{eq301} via a standard Borel-Cantelli argument.

To prove \eqref{eq301}, we choose a well-spaced net of points in the subcritical zone, and argue that for large $t$ with overwhelming probability, each ball of radius one centered at a net point has sufficiently many particles at time $t-1$ so that even simple Brownian motions initiated (rather than sub-BBMs) from the positions of these particles at time $t-1$ are enough to ensure that there is no ball of radius $r_t$ with center lying in the subcritical zone that remains not hit over the period $[t-1,t]$. In other words, the following occurs with overwhelming probability: over $[0,t-1]$, the system produces sufficiently many particles which are sufficiently well-spaced over the subcritical zone at time $t-1$, and then (neglecting the branching over $[t-1,t]$), Wiener sausages initiated from the positions of the particles at this time are enough to cover the entire subcritical zone. 

In this subsection, $d=2$. However, in some of the notation and arguments that follow, we prefer to keep $d$ general as they will be used in the next subsection as well, where $d\geq 3$.     

Let $\varepsilon>0$, and for $t>0$ let $\rho_t:=\sqrt{2\beta(1-\varepsilon)}t$ and $\mathbf{B}_t:=B(0,\rho_t)$. For $t>0$ define
$$m_t:=\left\lceil \frac{\rho_t}{1/(2\sqrt{d})} \right\rceil^d,
\quad n_t:= \left\lceil \frac{\rho_t}{r_t/(2\sqrt{d})} \right\rceil^d.$$
Then, $n_t=\left\lceil c_2\,t/r_t\right\rceil^d$ for some $c_2=c_2(\varepsilon,\beta,d)$. For $t>0$, define the events
$$A_t:=\left\{\textsf{vol}\left(\widehat{Z}_t^{r_t}\right)/t^2\leq 2\pi\beta(1-\varepsilon)\right\} .$$ 

First, we prepare the setting at time $t-1$. Let $C(0,\rho_t)$ be the cube centered at the origin with side length $2\rho_t$ so that $B(0,\rho_t)$ is inscribed in $C(0,\rho_t)$. Consider the simple cubic packing of $C(0,\rho_t)$ with balls of radius $1/(2\sqrt{d})$. Then, at most $m_t$ balls are needed to completely pack $C(0,\rho_t)$, say with centers $(x_j:1\leq j\leq m_t)$. For each $j$, let $B_j=B(x_j,1/(2\sqrt{d}))$. (We suppress the $t$-dependence in $x_j$ and $B_j$ for ease of notation.) In a simple cubic packing, the distance between a point in space and the farthest point of the packing ball that is closest to that point, is less than the distance between the center and any vertex of a $d$-dimensional cube with side length four times the radius of a packing ball. Then, since the distance between the center and any vertex of the $d$-dimensional unit cube is $\sqrt{d}/2$, it follows that
\begin{equation}
\forall\,x\,\in \mathbf{B}_t, \quad
\underset{1\leq j\leq m_t}{\min}\:\underset{z\in B_j}{\max}\:|x-z|<1. \label{eq41}
\end{equation}
For a Borel set $B\subseteq\mathbb{R}^d$ and $t\geq 0$, we write $Z_t(B)$ to denote the number of particles, i.e., the \emph{mass}, of $Z$ that fall inside $B$ at time $t$. For $j\in\{1,2,\ldots,m_t\}$, define the events
\begin{equation}
E_j:=\left\{Z_{t-1}(B_j)<e^{\beta (\varepsilon/2)t}\right\}. \label{events}
\end{equation}
Typically, the mass of BBM that fall inside a linearly moving ball of fixed radius $a>0$, say $B_t:=B(\theta\sqrt{2\beta}t\mathbf{e},a)$ for some unit vector $\mathbf{e}$ and $0<\theta<1$, is $\exp[\beta(1-\theta^2)t+o(t)]$. Moreover, \cite[Thm.\ 1]{M2018} says that for $0\leq a<1-\theta^2$, 
$$ \underset{t\rightarrow\infty}{\lim}\frac{1}{t}\log P\left(Z_t(B_t)<e^{\beta at}\right)=-\beta\times I$$   
for some positive rate function $I=I(\theta,a)$. Then, since $x_j\in B(0,\sqrt{2\beta(1-\varepsilon)}t)$ for each $j$, $\varepsilon/2$ is an atypically small exponent (typical exponent is at least $1-(\sqrt{1-\varepsilon})^2=\varepsilon$) for the mass of BBM in each $B_j$ at time $t-1$. It follows from \cite[Thm.\ 1]{M2018} that there exists a positive constant $c(\varepsilon)$ such that for all large $t$,
\begin{equation}
P\left(\bigcup_{1\leq j\leq m_t}E_j\right)\leq m_t e^{-c(\varepsilon) t}=e^{-c(\varepsilon) t+o(t)}, \label{eq42}
\end{equation}
where we have used the union bound and that $m_t$ is only a polynomial factor in $t$. It follows from \eqref{eq41}, \eqref{events} and \eqref{eq42} that at time $t-1$, with overwhelming probability, there are at least $e^{\beta (\varepsilon/2)t}$ particles in the $1$-neighborhood of each point in $\mathbf{B}_t$. That is, there exists $c=c(\varepsilon)>0$ such that for all large $t$,
\begin{equation}
P(G_t)\leq e^{-ct},\quad G_t:=\left\{\underset{x\in \mathbf{B}_t}{\inf}Z_{t-1}\left(B(x,1)\right)<e^{\beta(\varepsilon/2)t}\right\}. \label{eq43}
\end{equation}

Now consider the simple cubic packing of $C(0,\rho_t)$ with balls of radius $r_t/(2\sqrt{d})$. Then, at most $n_t$ balls are needed to completely pack $C(0,\rho_t)$, say with centers $(y_j:1\leq j\leq n_t)$. For each $j$, let $\widehat{B}_j=B(y_j,r_t/(2\sqrt{d}))$. By an argument similar to the one leading to \eqref{eq41}, it follows that 
\begin{equation}
\forall\,x\,\in \mathbf{B}_t, \quad
\underset{1\leq j\leq n_t}{\min}\:\underset{z\in \widehat{B}_j}{\max}\:|x-z|<r_t. \label{eq44}
\end{equation}
For $j\in\{1,2,\ldots,n_t\}$, define the events
$$F_j:=\{Z_s(\widehat{B}_j)=0\:\:\forall s\in[t-1,t] \} .$$ 
For $0\le t_1\le t_2$, let 
$$R(t_1,t_2):=\bigcup_{s=t_{1}}^{t_{2}} \mathrm{supp} (Z(s)),$$ 
that is, $R(t_1,t_2)$ is the accumulated support of $Z$ over $[t_1,t_2]$, and let  
$$\widehat{Z}_{[t_1,t_2]}^r:=\bigcup_{x\in R(t_1,t_2)}B(x,r) $$ 
be the corresponding sausage over $[t_1,t_2]$ with radius $r$. For $t>0$, define the events 
$$H_t:=\left\{\exists\:x\in \mathbf{B}_t\:\:\text{such that}\:\:x\notin \widehat{Z}_{[t-1,t]}^{r_t}\right\} .$$
Then, it follows from \eqref{eq44} that $H_t\subseteq\bigcup_{1\leq j\leq n_t}F_j$, and therefore 
$$P(H_t)\leq P\left(\bigcup_{1\leq j\leq n_t}F_j\right),$$
which holds even conditional on the event $G_t^c$. The union bound gives
\begin{equation}
P(H_t \mid G_t^c)\leq n_t P(F_k \mid G_t^c), \label{eq45}
\end{equation}
where $k$ is chosen so as to satisfy $P(F_k \mid G_t^c)=\underset{1\leq j\leq n_t}{\max}P(F_j \mid G_t^c)$. In view of \eqref{eq43}, \eqref{eq45}, and the estimate 
\begin{equation}P(A_t)\leq P(H_t)\leq P(H_t \mid G_t^c)+P(G_t), \nonumber
\end{equation}
and since $n_t$ is only an exponential factor in $t$, to complete the proof of \eqref{eq301}, it suffices to show that $P(F_k \mid G_t^c)$ is super-exponentially small in $t$ for large $t$.

Observe that conditional on the event $G_t^c$, the event $F_k$ can be realized only if the sub-BBMs initiated by each of the at least $\exp[\beta(\varepsilon/2)t]$ many particles present in $B(y_k,1)$ at time $t-1$ does not hit $\widehat{B}_k=B(y_k,r_t/(2\sqrt{d}))$ in the remaining time interval $[t-1,t]$. Apply the Markov property at time $t-1$, and neglect possible branching of particles over $[t-1,t]$ for an upper bound on $P(F_k \mid G_t^c)$. Then, by \eqref{hitting2} in Lemma~\ref{lemma2}, and the independence of particles present at time $t-1$, there exists $c>0$ such that for all large $t$,
\begin{equation}
P(F_k \mid G_t^c)\leq \left(1-\mathbf{P}_{\mathbf{e}}\left(\underset{0\leq s\leq 1}{\min}|X(s)|<\frac{r_t}{2\sqrt{d}}\right)\right)^{e^{\beta(\varepsilon/2)t}} \leq \left(1-\frac{c}{\beta kt}\right)^{e^{\beta(\varepsilon/2)t}}, \label{eq46}
\end{equation}
where $\mathbf{e}$ is a unit vector. It is clear that the right-hand side of \eqref{eq46} is super-exponentially small in $t$ for large $t$. This completes the proof of Theorem~\ref{theorem1}.

\subsection{Proof of Theorem~\ref{theorem2}}

We will show that for every $\varepsilon>0$ there exist positive constants $c_1$ and $c_2$ such that for all large $t$, 
\begin{equation} 
P\left(\frac{\textsf{vol}\left(\widehat{Z}_t^{r_t}\right)}{t^d}\leq [2\beta(1-k(d-2)-\varepsilon)]^{d/2}\omega_d\right)\leq e^{-c_1 t} ,\label{eq401}
\end{equation}
and
\begin{equation} 
P\left(\frac{\textsf{vol}\left(\widehat{Z}_t^{r_t}\right)}{t^d}\geq [2\beta(1-k(d-2)+\varepsilon)]^{d/2}\omega_d\right)\leq e^{-c_2 t}.\label{eq402}
\end{equation}
Then, Theorem~\ref{theorem2} will follow from \eqref{eq401} and \eqref{eq402} via a standard Borel-Cantelli argument.

The method of proof of \eqref{eq401} is identical to that of \eqref{eq301}. We only need to make the following changes. For $t>0$, let $\rho_t:=\sqrt{2\beta(1-k(d-2)-\varepsilon)}t$ and $\mathbf{B}_t:=B(0,\rho_t)$, and define the events $\left(A_t:t\geq 0\right)$ and $\left(E_j:j=1,2,\ldots,m_t\right)$ as
$$A_t:=\left\{\textsf{vol}(\widehat{Z}_t^{r_t})/t^d\leq [2\beta(1-k(d-2)-\varepsilon)]^{d/2}\omega_d \right\},\quad  E_j:=\left\{Z_{t-1}(B_j)<e^{\beta (k(d-2)+\varepsilon/2)t}\right\}.$$
Then, since $k(d-2)+\varepsilon/2$ is an atypically small exponent (typical exponent is at least $k(d-2)+\varepsilon$) for the mass in each $B_j$ at time $t-1$, by \cite[Thm.\ 1]{M2018}, there exists a positive constant $c(\varepsilon)$ such that for all large $t$, 
\begin{equation}
P\left(\bigcup_{1\leq j\leq m_t}E_j\right)\leq m_t e^{-c(\varepsilon) t}=e^{-c(\varepsilon) t+o(t)}. \nonumber
\end{equation}
The rest of the proof is identical to that of \eqref{eq301} except the last part, where we need to show that $P(F_k \mid G_t^c)$ is super-exponentially small in $t$ for large $t$. By \eqref{hitting3} in Lemma~\ref{lemma2}, and the independence of particles present at time $t-1$, there exists $c>0$ such that for all large $t$,
\begin{equation}
P(F_k \mid G_t^c)\leq \left(1-\mathbf{P}_{\mathbf{e}}\left(\underset{0\leq s\leq 1}{\min}|X(s)|<\frac{r_t}{2\sqrt{d}}\right)\right)^{e^{\beta (k(d-2)+\varepsilon/2)t}} \leq \left(1-ce^{-\beta k(d-2)t}\right)^{e^{\beta (k(d-2)+\varepsilon/2)t}}. \label{eq62}
\end{equation}
Using the inequality $1+x\leq e^x$, the right-hand side of \eqref{eq62} can be bounded from above by $\exp[-c\,e^{\beta(\varepsilon/2)t}]$, which is super-exponentially small in $t$ for large $t$. This completes the proof of \eqref{eq401}.

To prove \eqref{eq402}, for $0\leq\theta<1$, let $\mathcal{N}_t^\theta$ be the subset of $\mathcal{N}_t$ consisting of particles whose ancestral lines up to $t$ has exited $B_\theta:=B(0,\theta\sqrt{2\beta}t)$ at some point over $[0,t]$, and let $N_t^\theta:=|\mathcal{N}_t^\theta|$. Using a many-to-one formula, we have
\begin{equation} 
E[N_t^\theta]=E[N_t] \mathbf{P}_0\left(\underset{0\leq s\leq t}{\sup}|X(s)|>\theta\sqrt{2\beta}t\right). \label{eq304}
\end{equation}
Since $E[N_t]=e^{\beta t}$, Proposition A and \eqref{eq304} imply that 
\begin{equation} 
E[N_t^\theta]= e^{\beta(1-\theta^2)t+o(t)}. \label{eq306}
\end{equation}
Recall that $r_t=r_0 e^{-\beta kt}$. Each Brownian trajectory up to time $t$ traces out an associated $r_t$-sausage, with expected volume of $\kappa_{r_0} t\exp[-(d-2)\beta kt](1+o(1))$ by Lemma~\ref{lemma1}. Order the particles in $\mathcal{N}_t^\theta$, with respect to their birth times for instance, and let $\textsf{vol}(j)$ be the volume scooped out by the $r_t$-shrinking Wiener sausage associated to the ancestral line of $j$th particle in $\mathcal{N}_t^\theta$ over the time period $[\tau_j,t]$, where $\tau_j$ is the first time the ancestral line of that particle has exited $B_\theta$. Let $\widehat{Z}_t^{r_t,\theta}$ be the part of the sausage $\widehat{Z}_t^{r_t}$ that is outside $B_\theta$, that is, 
$\widehat{Z}_t^{r_t,\theta}:=\widehat{Z}_t^{r_t}\cap B_\theta^c$,
where $B_\theta^c$ denotes the complement of $B_\theta$ in $\mathbb{R}^d$. Then, by Lemma~\ref{lemma1} and \eqref{eq306},
\begin{align}
E\left[\textsf{vol}\left(\widehat{Z}_t^{r_t,\theta}\right)\right]&\leq \sum_{k=1}^\infty E[\textsf{vol}(1)+\ldots+\textsf{vol}(N_t^\theta)\mid N_t^\theta=k]P(N_t^\theta=k) \nonumber \\
&= \sum_{k=1}^\infty k\,P(N_t^\theta=k)\,E[\textsf{vol}(1)] \nonumber \\
&\leq E[N_t^\theta]\,\kappa_{r_0} te^{-(d-2)\beta kt}(1+o(1))=e^{\beta t(1-\theta^2-k(d-2)) +o(t)}. \label{eq307}
\end{align} 
To bound $E[\textsf{vol}(1)]$, we have applied the strong Markov property to the ancestral line of first particle of $\mathcal{N}_t^\theta$ at time $\tau_1$, which implies that $\textsf{vol}(1)$ is identically distributed as $\textsf{vol}(\cup_{0\leq s\leq t-\tau_1}B(X(s),r_t))$ given $\tau_1$. Let $\theta_1=\sqrt{1-k(d-2)+\varepsilon/2}$. It follows from the Markov inequality and \eqref{eq307} that
\begin{equation}
P(\textsf{vol}(\widehat{Z}_t^{r_t,\theta_1})\geq 1)\leq e^{-\beta \varepsilon t/2+o(t)}. \label{eq308}
\end{equation}
Observe that $\textsf{vol}(\widehat{Z}_t^{r_t})\leq \textsf{vol}(\widehat{Z}_t^{r_t,\theta_1})+(\theta_1\sqrt{2\beta}t+r_t)^d\omega_d< [2\beta(1-k(d-2)+\varepsilon)]^{d/2}\omega_d t^d$ for all large $t$ conditional on the event $\{\textsf{vol}(\widehat{Z}_t^{r_t,\theta_1})<1\}$. This, along with \eqref{eq308}, implies \eqref{eq402}, and hence completes the proof of Theorem~\ref{theorem2}.

\bibliographystyle{plain}

\end{document}